\documentclass[a4paper, 12pt]{amsart}

\usepackage[T1]{fontenc}
\usepackage[utf8]{inputenc}
\usepackage[left=3.5cm, right=3.5cm, top=3.75cm, bottom=3cm]{geometry}

\usepackage{amsmath, amssymb, amsthm}
\usepackage{chngcntr, cite, color}
\usepackage{dsfont}
\usepackage{enumitem, etoolbox}
\usepackage{hyperref}
\usepackage{indentfirst}

\hypersetup{colorlinks = true, linkcolor = black, filecolor = black, urlcolor = black, citecolor = black}

\makeatletter

    \renewcommand{\l@section}{\@tocline{1}{0pt}{1.5pc}{5pc}{}}
    \renewcommand{\l@subsection}{\@tocline{2}{0pt}{3.4pc}{0pc}{}}

    \newtheoremstyle{italicremark}
    {}{}
    {}
    {}
    {\itshape}
    {.}
    { }
    {\thmname{#1}\thmnumber{ #2}\thmnote{ (#3)}}
\makeatother

\newcounter{proofstep}
\AtBeginEnvironment{proof}{\setcounter{proofstep}0}

\newcounter{proofcase}
\AtBeginEnvironment{proof}{\setcounter{proofcase}0}

\newtheorem{theorem}{Theorem}[section]

\newtheorem{corollary}[theorem]{Corollary}
\newtheorem{lemma}[theorem]{Lemma}

\theoremstyle{definition}

\theoremstyle{italicremark}
\newtheorem{remark}[theorem]{\textit{Remark}}

\DeclareMathOperator{\intt}{int}

\DeclareMathOperator{\Ric}{Ric}

\newcommand{\Rbb}{\mathbb{R}}

\numberwithin{equation}{section}

\title[Topology and Dirichlet spectrum of free boundary surfaces]{Topology and Dirichlet spectrum of free boundary minimal and CMC surfaces}
\author[Alcides de Carvalho and Roney Santos]{Alcides de Carvalho and Roney Santos}

\address{Alcides de Carvalho \newline
\indent \normalfont Departamento de Matem\'atica \newline
\indent Universidade Federal de Pernambuco \newline
\indent Av. Jornalista An\'ibal Fernandes, 50740-560, Recife-PE, Brazil \newline
\indent \texttt{alcides.junior@ufpe.br}}

\address{Roney Santos \newline
\indent \normalfont Department of Mathematics \newline
\indent King's College London \newline
\indent Strand, WC2R 2LS, London, United Kingdom \newline
\indent \texttt{roney.santos@kcl.ac.uk}}
	
\keywords{Minimal surfaces, free boundary problems, Morse index, eigenvalue estimate}

\begin{document}

\begin{abstract}
    We find bounds for the genus and for the number of boundary components of a surface that is either index one as a free boundary minimal surface or stable as a free boundary constant mean curvature surface in a Riemannian three-manifold with convex boundary and non-negative Ricci curvature, provided the first Dirichlet eigenvalue of its Jacobi operator is non-negative. As an application, we establish strong topological restrictions on such surfaces in compact strictly convex domains of three-dimensional normal homogeneous spaces.
\end{abstract}

\maketitle

\section{Introduction}

Let $M$ be a Riemannian three-manifold with non-empty smooth boundary $\partial M$, and $\Sigma$ be a compact properly immersed surface in $M$ with non-empty smooth boundary $\partial \Sigma$. We say that $\Sigma$ is a \emph{free boundary minimal surface} if it is a critical point of the area functional among variations for which the boundary of $\Sigma$ is allowed to move along $\partial M$. Similarly, $\Sigma$ is a \emph{free boundary constant mean curvature surface} if it is a critical point of the same functional under a volume constraint.\smallskip

When the surface is two-sided, stability provides a natural link between the variational problem, the ambient geometry, and the topology of the surface. Hence, a substantial part of this theory has been devoted to understanding how curvature assumptions on the ambient manifold restrict the topology and geometry of free boundary minimal or constant mean curvature surfaces with low index or satisfying suitable stability conditions. For a possibly non-exhaustive list of works, see \cite{fraser2015Riemannsurfaces, barbosa2016stable, ambrozio2015rigidity, ambrozio2018index, medvedev2025geodesicballs, lima2026eigenvalue, sargent2017index, antonia2026topology, cavalcante2020index, mendes2018rigidity, ros1997capillary, ros2008stability, ros1995stability, nunes2017stable, wang2019uniqueness, batista2026first, tran2020index, barbosa2020liegroup} and the references therein.\smallskip

Our purpose in this paper is to relate the topology of a compact properly immersed surface with boundary and the Dirichlet spectrum of its Jacobi operator. More precisely, if $\mu_1(\Sigma)$ denotes the first Dirichlet eigenvalue of the Jacobi operator of a surface $\Sigma$ that is either minimal or constant mean curvature in a given Riemannian three-manifold, we prove the following.

\begin{theorem}\label{thm:main1}
    Let $\Sigma$ be a connected, two-sided and orientable surface that is either index one as a free boundary minimal surface or stable as a free boundary constant mean curvature surface in a Riemannian three-manifold $M$ with convex boundary and non-negative Ricci curvature. If $\mu_1(\Sigma) \geq 0$, then
    \begin{itemize}
        \item[$\mathrm{(i)}$] either $\Sigma$ has at most genus one and $\partial \Sigma$ has at most three components;
        \item[$\mathrm{(ii)}$] or $\Sigma$ is a minimal surface of genus two and $\partial \Sigma$ is a connected curve along which the second fundamental form of $\partial M$ vanishes.
    \end{itemize}
\end{theorem}

By slightly strengthening at least one of the geometric or spectral hypotheses in Theorem \ref{thm:main1}, we are able to rule out its second alternative. Moreover, by using an abstract generalization of a Nunes result \cite{nunes2017stable}, we are able to show that a surface as in Theorem \ref{thm:main1} cannot have genus two in a strictly convex domain of a normal homogeneous space.

\begin{theorem}\label{thm:main2}
    Let $\Sigma$ be a connected, two-sided and orientable surface that is either index one as a free boundary minimal surface or stable as a free boundary constant mean curvature surface in a compact domain $\Omega$ of a three-dimensional normal homogeneous space. Suppose $\partial \Omega$ is strictly convex. Then $\Sigma$ has at most genus one and $\partial \Sigma$ has at most three components.
\end{theorem}

Our Theorem \ref{thm:main2} extends the main result of \cite{barbosa2020liegroup} in two directions. First, the ambient space may be any three-dimensional normal homogeneous space, instead of a Lie group endowed with a bi-invariant metric, including for instance the product spaces $S^2 \times \Rbb$ and $S^2 \times S^1$. Second, besides stable free boundary constant mean curvature surfaces, the result also applies to free boundary minimal surfaces of index one. Therefore, the same topological restrictions hold in a more general geometric and variational setting.\smallskip

Other works have also obtained topological and rigidity results by using spectral conditions. For instance, Batista-Cavalcante-Melo \cite{batista2026first} considered Robin and Jacobi-Steklov spectral assumptions, while Antonia-Cavalcante-Souza \cite{antonia2026topology} obtained topological estimates for stable free boundary constant mean curvature surfaces under lower Ricci curvature bounds through a constrained Robin spectral problem. In contrast, our approach uses directly the Dirichlet spectrum of the Jacobi operator, includes the borderline case of a vanishing first Dirichlet eigenvalue, and gives restrictions for the genus and the number of boundary components of both index one free boundary minimal surfaces and stable free boundary constant mean curvature surfaces under convexity of the boundary and non-negative Ricci curvature.

\subsubsection*{LLM disclosure}

Initially, we intended to apply Theorem \ref{thm:main1} to strictly convex domains of simply connected space forms of non-negative sectional curvature. But ChatGPT pointed out that normal homogeneous spaces have the properties needed for the argument, which led us to Theorem \ref{thm:main2}. We also used ChatGPT to revise some definitions and results, as well as to review the manuscript.

\subsubsection*{Acknowledgments}

R. Santos was supported by FAPESP grant 2025/03506-0 and CNPq grant 409513/2023-7.

\section{Proof of Theorem \ref{thm:main1}}

Let $M$ be a Riemannian three-manifold and $\Sigma$ an immersed surface in $M$. Consider a local unit normal vector field $N$ along $\Sigma$. We define the second fundamental form of $\Sigma$ as $A_\Sigma(X, Y) = -\langle \nabla_XN, Y\rangle$, where $X$ and $Y$ are vector fields tangent to $\Sigma$, and its mean curvature as $H_\Sigma = \operatorname{tr}(A_\Sigma)$. The surface $\Sigma$ is \emph{minimal} if $H_\Sigma$ vanishes identically, and \emph{constant mean curvature} if $H_\Sigma$ is constant along it. If $\partial M$ is non-empty and smooth, it is a surface in $M$ that we call \emph{convex} if $A_{\partial M} \geq 0$ with respect to its inward unit normal, and \emph{strictly convex} if the inequality is strict.\smallskip

We aim to consider a connected and compact properly immersed surface $\Sigma$ whose non-empty smooth boundary $\partial \Sigma$ meets $\partial M$ orthogonally. It is called a \emph{free boundary minimal surface} if $H_\Sigma$ vanishes, and a \emph{free boundary constant mean curvature surface} if $H_\Sigma$ is constant. The first variation formula shows that they are critical points of the area functional among variations that keep $\partial \Sigma$ in $\partial M$, in the second case under a volume constraint.\smallskip

This means, in particular, that one can consider only variational vector fields of the form $X = fN$, where $f \in C^\infty(\Sigma)$ and $N|_{\partial \Sigma}$ is tangent to $\partial M$. The second variation formula, then, defines the quadratic form
\[Q_\Sigma(f) = \int_\Sigma\big(|\nabla f|^2 - (\Ric(N, N) + |A_\Sigma|^2)f^2\big)\, dA - \int_{\partial \Sigma} A_{\partial M}(N, N)f^2\, ds,\]
where $\Ric$ is the Ricci curvature of $M$, with associated symmetric bilinear form
\[Q_\Sigma(f, h) = \int_\Sigma \big(\langle\nabla f, \nabla h\rangle - (\Ric(N,N) + |A_\Sigma|^2)fh\big)\, dA - \int_{\partial\Sigma} A_{\partial M}(N,N) fh\, ds.\]
Integration by parts gives that
\begin{equation}\label{eq:indexform}
    Q_\Sigma(f, h) = -\int_\Sigma h\, L_\Sigma(f)\, dA + \int_{\partial \Sigma} h \big(\nu(f) - A_{\partial M}(N, N)f\big)\, ds,
\end{equation}
where $L_\Sigma = \Delta + \Ric(N, N) + |A_\Sigma|^2$ is the \emph{Jacobi operator} of $\Sigma$. Since $\Sigma$ is compact and $L_\Sigma$ is a second-order elliptic operator with smooth coefficients, the \emph{Robin eigenvalue problem}
\[\left\{
    \begin{aligned}
        L_\Sigma(f) + \lambda f &= 0\,\mbox{ in }\,\Sigma\\
        \nu(f) - A_{\partial M}(N, N)f &= 0\,\mbox{ on }\,\partial\Sigma.
    \end{aligned}
\right.\]
defines a self-adjoint elliptic problem with compact resolvent. Hence its spectrum is real and discrete, with eigenvalues $\lambda_1(\Sigma) \leq \lambda_2(\Sigma) \leq \ldots \lambda_k(\Sigma) \to +\infty$. In particular, it has only finitely many negative eigenvalues. The \emph{index} of $\Sigma$ is the number of negative eigenvalues, counted with multiplicities, of this Robin problem. Moreover, it is well-known that the \emph{first Robin eigenfunction} is strictly positive.\smallskip

If $\Sigma$ has constant mean curvature, the volume constraint implies that $f$ has mean value zero. Hence, in that case, the surface $\Sigma$ is \emph{stable} if and only if $Q_\Sigma(f) \geq 0$ for all $f \in C^\infty(\Sigma)$ with zero mean value.\smallskip

In order to use the stability condition, we need a suitable class of test functions, as the ones used by Nunes \cite{nunes2017stable} obtained from the Ahlfors-Gabard theorem \cite{ahlfors1950open, gabard2006representation} with a Hersch type balancing argument \cite{hersch1970quatre}. Since we could not find such a statement, we include it here together with a verification for the sake of completeness.

\begin{lemma}\label{lem:hersch}
    Let $\Sigma$ be a compact smooth surface with smooth boundary $\partial \Sigma$ and $F: \Sigma \to S^2_+$ be a smooth map with $F(\intt(\Sigma)) \subset \intt(S^2_+)$ and $F(\partial \Sigma) \subset \partial S^2_+$. Let $\rho: \Sigma \to \Rbb$ be a positive continuous function. Then there exists a conformal diffeomorphism $T: S^2_+ \to S^2_+$ such that, if $T \circ F = (f_1, f_2, f_3)$, one has
    \[\int_\Sigma \rho f_1\, dA = \int_\Sigma \rho f_2\, dA = 0\]
    and $f_3 = 0$ along $\partial \Sigma$.
\end{lemma}

\begin{proof}
    Write points of $S^2$ as $y = (x, x_3)$, with $x \in D^2$, where $D^2$ is the unit disk centered at the origin of $\Rbb^2$. For each $v \in \intt(D^2)$, define the conformal transformation $T_v: S^2 \to S^2$ by
    \[T_v(x,x_3) = \left(\frac{(1 - |v|^2)x + 2(1 + \langle v, x\rangle)v}{1 + 2\langle v, x\rangle + |v|^2}, \frac{(1 - |v|^2)x_3}{1 + 2\langle v, x\rangle + |v|^2}\right).\]
    This is the restriction to $S^2$ of a Möbius transformation. Since the last coordinate of $T_v$ is non-negative for $x_3 \geq 0$, we get $T_v(S^2_+) \subset S^2_+$. Since $T_v^{-1} = T_{-v}$ satisfies the same property, it follows that $T_v(S^2_+) = S^2_+$. Define
    \[Y(v) = \int_\Sigma \rho \cdot (f_1^v, f_2^v)\, dA,\]
    where $T_v \circ F = (f_1^v, f_2^v, f_3^v)$. Suppose $Y(v) \neq 0$ for all $v \in \intt(D^2)$. Then
    \[G: \intt(D^2) \to S^1,\quad G(v) = \frac{Y(v)}{|Y(v)|},\]
    is continuous. Let $v_0 \in S^1$. Since $|v_0| = 1$, as $v \to v_0$, the Möbius transformation $T_v$ converges pointwise, away from the antipodal point $(-v_0, 0)$, to the constant map $y \mapsto (v_0, 0)$. Since $F$ is proper, $ F^{-1}((-v_0,0)) \subset \partial\Sigma$ has zero area measure, and thus the convergence holds almost everywhere on $\Sigma$. Moreover, $\rho \cdot \big((f_1^v)^2 + (f_2^v)^2\big) \leq \rho$, and $\rho$ is integrable because it is continuous and $\Sigma$ is compact. Thus, the dominated convergence theorem gives
    \[\lim_{v \to v_0}Y(v) = \left(\int_\Sigma \rho\, dA\right)v_0.\]
    As a consequence, one finds a continuous extension $\widehat{G}: D^2 \to S^1$ of $G$ such that $\widehat{G}|_{S^1}$ is the identity map of $S^1$, which is a contradiction. We, thus, conclude that there exists $v \in D^2$ such that $Y(v) = 0$. Taking $T = T_v|_{S^2_+}$ we obtain
    \[\int_\Sigma \rho f_1^v\, dA = \int_\Sigma \rho f_2^v\, dA = 0.\]
    Finally, since $F(\partial \Sigma) \subset \partial S^2_+$, we have that $f_3 = 0$ along $\partial \Sigma$.
\end{proof}

We consider the \emph{Dirichlet eigenvalue problem} for $L_\Sigma$, which reads as
\[\left\{
\begin{aligned}
    L_\Sigma(f) + \mu f &= 0\,\mbox{ in }\,\Sigma\\
    f &= 0\,\mbox{ on }\,\partial\Sigma.
\end{aligned}
\right.\]
Its first eigenvalue is then given by
\[\mu_1(\Sigma) = \inf_{0 \neq f \in H^1_0(\Sigma)} \frac{Q_\Sigma(f)}{\int_\Sigma f^2\, dA},\]
where
\[H_0^1(\Sigma) = \{f\in H^1(\Sigma):\, f = 0 \text{ on }\partial\Sigma\}\]
and
\[H^1(\Sigma) = \{f\in L^2(\Sigma):\, \nabla f\in L^2(T\Sigma)\}.\]
Note that $\mu_1(\Sigma) \geq 0$ if and only if $Q_\Sigma(f) \geq 0$ for all $f \in H^1_0(\Sigma)$.\smallskip

In the results below, we denote by $g(\Sigma)$ the genus of $\Sigma$ and by $b(\partial \Sigma)$ the number of connected components of $\partial \Sigma$.

\begin{theorem}
    Let $\Sigma$ be a connected, two-sided and orientable surface that is either index one as a free boundary minimal surface or stable as a free boundary constant mean curvature surface in a Riemannian three-manifold $M$ with convex boundary and non-negative Ricci curvature. If $\mu_1(\Sigma) \geq 0$, then
    \begin{itemize}
        \item[$\mathrm{(i)}$] either $g(\Sigma) \leq 1$ and $b(\partial \Sigma) \leq 3$;
        \item[$\mathrm{(ii)}$] or $g(\Sigma) = 2$, $b(\partial \Sigma) = 1$, $H_\Sigma = 0$ and $A_{\partial M} = 0$ along $\partial \Sigma$.
    \end{itemize}
\end{theorem}

\begin{proof}
    Let $g$ be the genus of $\Sigma$ and let $b$ be the number of connected components of $\partial \Sigma$. Suppose $\Sigma$ is a stable free boundary constant mean curvature surface. By the Ahlfors-Gabard theorem \cite{ahlfors1950open, gabard2006representation}, there exists a proper conformal branched cover $F = (f_1,f_2,f_3): \Sigma \to S^2_+$ of degree $d \leq g + b$. Composing it with the conformal map of Lemma \ref{lem:hersch} and keeping the same notation, its degree does not change and, by taking $\rho = 1$, we may assume that
    \[\int_\Sigma f_1\, dA = \int_\Sigma  f_2\, dA = 0\]
    and $f_3 \in H_0^1(\Sigma)$. By stability, $Q_\Sigma(f_1) \geq 0$ and $Q_\Sigma(f_2) \geq 0$. Since $\mu_1(\Sigma) \geq 0$, then $Q_\Sigma(f_3) \geq 0$. Therefore
    \begin{align*}
        0 &\leq Q_\Sigma(f_1) + Q_\Sigma(f_2) + Q_\Sigma(f_3)\phantom{\int}\\
        &= \int_\Sigma |\nabla F|^2\, dA - \int_\Sigma\big(\Ric(N, N) + |A_\Sigma|^2\big)\, dA - \int_{\partial\Sigma} A_{\partial M}(N, N)\, ds.
    \end{align*}
    Since $F$ is conformal and has degree $d$,
    \[\int_\Sigma|\nabla F|^2\, dA = 4\pi d.\]
    Hence
    \begin{equation}\label{eq:inequality1}
        \int_\Sigma\big(\Ric(N, N) + |A_\Sigma|^2\big)\, dA + \int_{\partial\Sigma} A_{\partial M}(N, N)\, ds \leq 4\pi d.
    \end{equation}

    Let $\{e_1, e_2\}$ be a local orthonormal frame tangent to $\Sigma$. It follows from the Gauss equation that
    \begin{equation}\label{eq:inequality4}
        \Ric(N, N) + |A_\Sigma|^2 + 2K_\Sigma - H_\Sigma^2 = \Ric(e_1, e_1) + \Ric(e_2, e_2) \geq 0.
    \end{equation}
    Thus, from \eqref{eq:inequality1},
    \begin{equation}\label{eq:inequality2}
        \int_\Sigma H_\Sigma^2\, dA - 2\int_\Sigma K_\Sigma\, dA + \int_{\partial\Sigma}A_{\partial M}(N, N)\, ds \leq 4\pi d.
    \end{equation}
    On the other hand, since $\Sigma$ meets $\partial M$ orthogonally, the geodesic curvature of $\partial\Sigma$ in $\Sigma$ is $\kappa = A_{\partial M}(T, T)$, where $T$ is a unit tangent vector to $\partial \Sigma$. Hence, by the Gauss-Bonnet theorem,
    \[\int_\Sigma K_\Sigma\, dA + \int_{\partial\Sigma}A_{\partial M}(T, T)\, ds = 2\pi(2 - 2g - b).\]
    Combining this identity with \eqref{eq:inequality2}, we obtain
    \[\int_\Sigma H_\Sigma^2\, dA + \int_{\partial\Sigma} \big(A_{\partial M}(N, N) + 2A_{\partial M}(T, T)\big)\, ds \leq 4\pi (d + 2 - 2g - b).\]
    Since $d \leq g + b$,
    \begin{equation}\label{eq:inequality3}
        \int_\Sigma H_\Sigma^2\, dA + \int_{\partial\Sigma}\big(A_{\partial M}(N, N) + 2A_{\partial M}(T, T)\big)\, ds \leq 4\pi(2 - g).
    \end{equation}
    Since $A_{\partial M} \geq 0$, we get $g \leq 2$, and now the bounds for $b$ follow directly from item (i) of Theorem 1.2 in \cite{fraser2015Riemannsurfaces} (see Remark \ref{rmk:cfp} below). If $g = 2$, then inequality \eqref{eq:inequality3} becomes an equality, and it implies that $H_\Sigma$ vanishes on $\Sigma$ and $A_{\partial M}$ vanishes on $\partial \Sigma$.\smallskip

    Now, assume that $\Sigma$ is an index one free boundary minimal surface, and denote by $\rho_1$ a positive first Robin eigenfunction of the Jacobi operator $L_\Sigma$. By the Ahlfors-Gabard theorem, there exists a proper conformal branched cover $F = (f_1,f_2,f_3): \Sigma \to S^2_+$ of degree $d \leq g + b$, which we can balance with Lemma \ref{lem:hersch} in order to obtain that
    \[\int_\Sigma \rho_1 f_1\,dA = \int_\Sigma \rho_1 f_2\,dA = 0\]
    and $f_3 \in H_0^1(\Sigma)$ so that $Q_\Sigma(f_1) \geq 0$ and $Q_\Sigma(f_2) \geq 0$ since $\Sigma$ has index one, and $Q_\Sigma(f_3) \geq 0$ because $\mu_1(\Sigma) \geq 0$. The rest of the argument is analogous.
\end{proof}

\begin{remark}\label{rmk:cfp}
    We have applied Theorem 1.2 of \cite{fraser2015Riemannsurfaces} to the constant mean curvature case, although it is stated for index one free boundary minimal surfaces. However, the same argument with $h = 1$ instead of the first Robin eigenfunction of $L_\Sigma$ yields exactly the same result.
\end{remark}

It follows directly from inequalities \eqref{eq:inequality4} and \eqref{eq:inequality3} that the genus two case can be easily ruled out by imposing restrictions on the geometry of the ambient manifold.

\begin{corollary}
    Let $\Sigma$ be a connected, two-sided and orientable surface that is either index one as a free boundary minimal surface or stable as a free boundary constant mean curvature surface in a Riemannian three-manifold $M$ with smooth boundary such that
    \begin{itemize}
        \item[$\mathrm{(i)}$] either $\partial M$ is strictly convex and $M$ has non-negative Ricci curvature;
        \item[$\mathrm{(ii)}$] or $\partial M$ is convex and $M$ has positive Ricci curvature.
    \end{itemize}
    If $\mu_1(\Sigma) \geq 0$, then $g(\Sigma) \leq 1$ and $b(\partial \Sigma) \leq 3$.
\end{corollary}

Also, item (ii) of Theorem \ref{thm:main1} cannot be true by a spectral restriction on $\Sigma$. Indeed, if $\mu_1(\Sigma) > 0$, one has $Q_\Sigma(f_3) > 0$ and thus the inequalities \eqref{eq:inequality1}, \eqref{eq:inequality2} and \eqref{eq:inequality3} become strict.

\begin{corollary}
    Let $\Sigma$ be a connected, two-sided and orientable surface that is either index one as a free boundary minimal surface or stable as a free boundary constant mean curvature surface in a Riemannian three-manifold $M$ with convex boundary and non-negative Ricci curvature. If $\mu_1(\Sigma) > 0$, then $g(\Sigma) \leq 1$ and $b(\partial \Sigma) \leq 3$.
\end{corollary}

Finally, when $\Sigma$ has non-zero constant mean curvature, it follows from \eqref{eq:inequality3} that its genus must be at most one.

\begin{corollary}
    Let $\Sigma$ be a connected, two-sided and orientable stable free boundary constant mean curvature surface in a Riemannian three-manifold $M$ with convex boundary and non-negative Ricci curvature. If $\mu_1(\Sigma) \geq 0$ and $H_\Sigma \neq 0$, then $g(\Sigma) \leq 1$ and $b(\partial \Sigma) \leq 3$.
\end{corollary}

\section{Proof of Theorem \ref{thm:main2}}

In \cite{nunes2017stable}, Nunes shows that, for every stable free boundary constant mean curvature surface in a strictly convex domain of the Euclidean space one has $Q_\Sigma(f) \geq 0$ for all $f \in H^1_0(\Sigma)$. As a consequence, $\mu_1(\Sigma) \geq 0$. The proof relies on the existence of a Jacobi function on the surface that is constructed by means of the Euclidean structure and that is a negative direction of $Q_\Sigma$. Using this idea, we are able to generalize his result in the setting of index one free boundary minimal surfaces, as well as one by Tran \cite{tran2020index}.

\begin{lemma}[]\label{lem:nunes-tran1}
    Let $\Sigma$ be a connected and two-sided index one free boundary minimal surface in a Riemannian three-manifold $M$. Then $\mu_1(\Sigma) > 0$ if and only if there exists $\phi \in C^\infty(\Sigma)$ such that $L_\Sigma(\phi) = 0$ and $Q_\Sigma(\phi) < 0$.
\end{lemma}

\begin{proof}
    Suppose $\mu_1(\Sigma) > 0$. Then there exists $f \in C^\infty(\Sigma)$ with $Q_\Sigma(f) < 0$ such that the Dirichlet problem
    \[\left\{
    \begin{aligned}
        L_\Sigma(\phi) &= 0\,\mbox{ in }\,\Sigma\\
        \phi &= f\,\mbox{ on }\,\partial\Sigma
    \end{aligned}
    \right.\]
    admits a unique solution. Set $u = f - \phi$. Then $u \in H_0^1(\Sigma)$, and consequently
    \[Q_\Sigma(u) \geq \mu_1(\Sigma) \int_\Sigma u^2\, dA \geq 0.\]
    Since $L_\Sigma(\phi) = 0$, identity \eqref{eq:indexform} shows that $Q_\Sigma(u, \phi) = 0$, and therefore
    \[Q_\Sigma(\phi) = Q_\Sigma(f) - Q_\Sigma(u) < 0.\]

    Now, suppose there exists $\phi \in C^\infty(\Sigma)$ such that $L_\Sigma(\phi) = 0$ and $Q_\Sigma(\phi) < 0$. Let $\rho_1 > 0$ be a first eigenfunction of the Robin problem for $L_\Sigma$. Since $\Sigma$ is index one,
    \[\int_\Sigma \rho_1 \phi\, dA \neq 0.\]
    Suppose, by contradiction, that $\mu_1(\Sigma) \leq 0$, and let $\psi_1$ be a non-negative first Dirichlet eigenfunction of $L_\Sigma$. Then $Q_\Sigma(\psi_1) \leq 0$, and since $\psi_1 > 0$ in $\intt(\Sigma)$,
    \[\int_\Sigma \rho_1 \psi_1\, dA > 0.\]
    Multiplying $\phi$ by a constant, we may assume that
    \[\int_\Sigma \rho_1 \psi_1\, dA = \int_\Sigma \rho_1 \phi\, dA.\]
    In particular, the function $f = \psi_1 - \phi$ is orthogonal to $\rho_1$. Since $L_\Sigma(\phi) = 0$ and $\psi_1 = 0$ along $\partial \Sigma$, we must have $Q_\Sigma(\psi_1, \phi) = 0$ by identity \eqref{eq:indexform}, and thus
    \[Q_\Sigma(f) = Q_\Sigma(\psi_1) + Q_\Sigma(\phi) < 0,\]
    contradicting that $\Sigma$ is index one. As a result, $\mu_1(\Sigma) > 0$.
\end{proof}

We also extend the results of Nunes \cite{nunes2017stable} and Tran \cite{tran2020index} in the constant mean curvature setting. But, in this case, we need to impose suitable geometric restrictions on the ambient Riemannian three-manifold in order to find at least one direction on which the index form is negative.

\begin{lemma}[]\label{lem:nunes-tran2}
    Let $\Sigma$ be a connected and two-sided stable free boundary constant mean curvature surface in a Riemannian three-manifold $M$ with smooth boundary such that
    \begin{itemize}
        \item[$\mathrm{(i)}$] either $\partial M$ is strictly convex and $M$ has non-negative Ricci curvature;
        \item[$\mathrm{(ii)}$] or $\partial M$ is convex and $M$ has positive Ricci curvature.
    \end{itemize}
    Then $\mu_1(\Sigma) > 0$ if and only if there exists $\phi \in C^\infty(\Sigma)$ such that $L_\Sigma(\phi) = 0$ and $Q_\Sigma(\phi) < 0$.
\end{lemma}

\begin{proof}
    Suppose $\mu_1(\Sigma) > 0$. Then for every $f \in C^\infty(\Sigma)$ such that $Q_\Sigma(f) < 0$ the Dirichlet problem
    \[\left\{
    \begin{aligned}
        L_\Sigma(\phi) &= 0\,\mbox{ in }\,\Sigma\\
        \phi &= f\,\mbox{ on }\,\partial\Sigma
    \end{aligned}
    \right.\]
    admits a unique solution. Take $f = 1$, and set $u = 1 - \phi$. Then $u \in H_0^1(\Sigma)$, and consequently
    \[Q_\Sigma(u) \geq \mu_1(\Sigma) \int_\Sigma u^2\, dA \geq 0.\]
    Since $L_\Sigma(\phi) = 0$, identity \eqref{eq:indexform} shows that $Q_\Sigma(u, \phi) = 0$, and therefore
    \[Q_\Sigma(\phi) = Q_\Sigma(1) - Q_\Sigma(u) < 0\]
    by the geometric hypotheses on $M$ and $\partial M$.\smallskip
    
    The converse statement is proved as in Lemma \ref{lem:nunes-tran1} by replacing the first Robin eigenfunction $\rho_1$ with the constant function $1$.
\end{proof}

Let $G$ be a Lie group, $H \subset G$ a closed subgroup, and let $\mathfrak{g}$ and $\mathfrak{h}$ denote their Lie algebras, respectively. Suppose that $\mathfrak{g}$ is endowed with an $\operatorname{Ad}(G)$-invariant inner product $\langle \cdot, \cdot\rangle$. Writing $\mathfrak{g} = \mathfrak{h} \oplus \mathfrak{h}^{\perp}$, we identify $\mathfrak{h}^{\perp}$ with $T_oM$, where $M = G/H$ and $o = eH$. The restriction of $\langle\cdot, \cdot\rangle$ to $\mathfrak{h}^{\perp}$ induces a $G$-invariant Riemannian metric on $M$. A homogeneous space endowed with a metric obtained in this way is called a \emph{normal homogeneous space}.\smallskip

Every normal homogeneous space has non-negative sectional curvature. Indeed, if $X, Y \in \mathfrak{h}^{\perp} \simeq T_oM$ are orthonormal, then
\[K_M(X,Y) = |[X,Y]_{\mathfrak{h}}|^2 + \frac{1}{4}|[X,Y]_{\mathfrak{h}^\perp}|^2 \geq 0.\]
In particular, $\Ric \geq 0$.\smallskip

Let $\{E_1, \ldots, E_n\}$ be an orthonormal basis of $\mathfrak{g}$, and let $X_\alpha$ be the fundamental vector field on $M$ associated with $E_\alpha$, namely
\[X_\alpha(p) = \frac{d}{dt}\Big|_{t=0} \exp(tE_\alpha)\cdot p.\]
Since the action of $G$ on $M$ is by isometries, each $X_\alpha$ is a Killing vector field. These Killing vector fields satisfy the Parseval formula
\[\sum_{\alpha = 1}^n \langle X_\alpha(p), v\rangle^2 = |v|^2\]
for every $p \in M$ and every $v \in T_pM$.\smallskip

We now prove the following.

\begin{theorem}
    Let $\Sigma$ be a connected, two-sided and orientable surface that is either index one as a free boundary minimal surface or stable as a free boundary constant mean curvature surface in a compact domain $\Omega$ of a three-dimensional normal homogeneous space $M = G/H$. Suppose $\partial \Omega$ is strictly convex. Then $g(\Sigma) \leq 1$ and $b(\partial \Sigma) \leq 3$.
\end{theorem}

\begin{proof}
    Let $\{E_1, \ldots, E_n\}$ be an orthonormal basis of the Lie algebra of $G$, and let $X_\alpha$ be the fundamental vector field on $M$ associated with $E_\alpha$. For each $\alpha$, define $\phi_\alpha = \langle X_\alpha, N\rangle$, where $N$ is a globally defined unit normal vector field along $\Sigma$. Since $X_\alpha$ is Killing and $\Sigma$ has constant mean curvature, $\phi_\alpha$ satisfies $L_\Sigma(\phi_\alpha) = 0$. Moreover, by the Parseval formula, one has $\sum_{\alpha = 1}^n\phi_\alpha^2 = 1$, and hence
    \begin{align*}
        \sum_{\alpha = 1}^n Q_\Sigma(\phi_\alpha) &= \frac{1}{2} \int_{\partial\Sigma} \nu \Big(\sum_{\alpha = 1}^n \phi_\alpha^2 \Big)\, ds - \int_{\partial\Sigma} A_{\partial\Omega}(N, N) \sum_{\alpha=1}^n\phi_\alpha^2\,ds\\
        &= -\int_{\partial\Sigma} A_{\partial\Omega}(N, N)\, ds.
    \end{align*}
    Consequently, since $\partial \Omega$ is strictly convex, then $Q_\Sigma(\phi_\alpha) < 0$ for at least one value of $\alpha \in\{1, \ldots, n\}$. It follows from Lemmas \ref{lem:nunes-tran1} and \ref{lem:nunes-tran2} that $\mu_1(\Sigma) > 0$, and therefore Theorem \ref{thm:main1} implies $g \leq 1$ and $b \leq 3$.
\end{proof}

\bibliographystyle{acm}
\bibliography{references}
	
\end{document}